\documentclass[reqno,12pt,letterpaper]{amsart}

\usepackage[colorlinks=true,linkcolor=Red,citecolor=Green]{hyperref}
\usepackage[usenames,dvipsnames]{xcolor}
\usepackage{paralist}
\usepackage{standard}

\numberwithin{equation}{section}

\newtheorem{prop}{Proposition}[section]
\newtheorem{lemm}[prop]{Lemma}

\theoremstyle{definition}
\newtheorem{defi}[prop]{Definition}

\newcommand{\hyp}{\mathcal{H}}
\newcommand{\gl}{\mathcal{G}}
\newcommand{\ellip}{\mathcal{E}}
\newcommand{\ham}{\mathsf{H}}
\newcommand{\diam}{\operatorname{diam}}

\newcommand{\loc}{\mathrm{loc}}

\renewcommand{\WF}{\mathrm{WF}_h}

\newcommand{\GBB}{\mathrm{GBB}}

\newcommand{\IC}{\mathsf{IC}}

\title[Diffractive trapping]{Resonance-free regions for diffractive trapping by
  conormal potentials}

\author{Oran Gannot and Jared Wunsch}
\address{Department of Mathematics, Lunt Hall, Northwestern University,
	Evanston, IL 60208, USA}
\email{gannot@northwestern.edu \and jwunsch@math.northwestern.edu}
      
\begin{document}

\begin{abstract}
  We consider the Schr\"odinger operator
  \[
P=h^2 \Lap_g +V
\]
on $\RR^n$ equipped with a metric $g$ that is Euclidean outside a
compact set.  The real-valued potential $V$ is assumed to be compactly
supported and smooth except at \emph{conormal
  singularities} of order $-1-\alpha$ along a compact hypersurface $Y.$
For $\alpha>2$ (or even $\alpha>1$ if the classical flow is unique),
we show that if $E_0$ is a non-trapping energy for the classical flow, then the operator $P$ has no
resonances in  a region
	\[
 [E_0 - \delta, E_0 + \delta] - i[0,\nu_0 h \log(1/h)].
      \]
     The constant $\nu_0$ is explicit in terms of $\alpha$ and
      dynamical quantities. We also show that the size of this resonance-free region is optimal for the class of piecewise-smooth potentials on the line.
\end{abstract}

\maketitle

\section{Introduction}

\subsection{Main results}
Let $X = \RR^n$, equipped with a smooth Riemannian metric $g$ such that $g_{ij} = \delta_{ij}$ outside a compact set. With $\Lap_g$ denoting the nonnegative Laplacian on $(X,g)$, consider the semiclassical Schr\"odinger operator with compactly supported potential,
\[
P = h^2\Lap_g + V.
\]
We assume that $V \in I^{[-1-\alpha]}(Y)$ is conormal to a compact hypersurface $Y \subset X$ with $\alpha > 1$. This notation means the following:
\begin{itemize} \itemsep6pt
	\item $V$ is $\CI$ away from $Y$.
	\item In local coordinates $(x^1,x')$ near $Y$, with $Y$ given by $\{x^1 = 0 \}$,
\[
V(x^1,x') = \int e^{ i x^1\xi^1} v(x,\xi^1) \, d\xi^1, \quad v \in S^{-1-\alpha}(\RR^n_x; \RR_{\xi^1}).
\]
\end{itemize}
Such a potential $V$ is thus $\alpha$ derivatives smoother than a step
function discontinuity across the interface $Y$, and by taking $\alpha>1$
we ensure that is at least $\mathcal{C}^{1,\gamma}$ for some $\gamma > 0$. Let $p = \sigma_h(P)$ denote the semiclassical principal symbol of $P$,
\[
p = |\xi|^2_g + V.
\]
The Hamilton vector field $\ham_p$ is continuous, hence always has global solutions. We further \emph{assume} that $\ham_p$ has unique integral curves; this is always true if $\alpha > 2$ (where $\ham_p$ is Lipschitz) but in general fails in the range $\alpha \in (1,2]$. In particular, there is a well-defined flow 
\[
\rho \mapsto \exp_{t\ham_p}(\rho)
\]
on $T^*X$, which is tangent to each energy surface $\{ p = E\}$.

Let $E_0 > 0$ be a \emph{non-trapping} energy level for the $\ham_p$
flow, i.e., assume that $\abs{x}\to \infty$ in both directions along all integral curves
of $\ham_p$ in $\{p=E_0\}.$ We show that for a suitable $\nu_0 > 0$ and $h,\delta> 0$ sufficiently small, there are no resonances of $P$ in the spectral window
\[
[E_0 - \delta, E_0 + \delta] - i[0,\nu_0 h \log(1/h)].
\]
The quantity $\nu_0 > 0$ has a dynamical characterization which we
discuss next. Let $\hyp_E \subset T^*Y$ denote the set of hyperbolic
points for $p - E;$ these are the points in phase space where the flow
is transverse
to $Y.$  More precisely, if we introduce normal coordinates $(x^1,x')$ for $g$ with respect to $Y$, so that $Y = \{x^1=0\}$ locally, then we can write
\begin{align*}
p(x,\xi) - E &= (\xi^1)^2 + \langle K(x)\xi',\xi'\rangle + V(x) - E \\
& \equiv (\xi^1)^2 - r(x,\xi',E)
\end{align*}
for a positive definite matrix $K(x)$. In these coordinates, 
\[
(x',\xi') \in \hyp_E \Longleftrightarrow r(0,x',\xi',E) > 0.
\]
We also remark for later use that the glancing set $\gl_E \subset T^*Y$ is defined in coordinates by the equation 
\[
r(0,x',\xi',E) = 0.
\]
For $x \in Y$, let $\pi: T^*_x X \rightarrow T^*_x Y$ denote the
canonical projection, which in local coordinates is just the map
$(0,x',\xi^1,\xi') \mapsto (x',\xi')$.  Note that $\pi$ is two-to-one over
$\hyp_E$ and one-to-one over $\gl_E.$

Given $E \in \RR$, we introduce the affine length of the longest $\ham_p$ trajectory connecting two hyperbolic points:
\begin{equation} \label{eq:diameter}
\diam_E(Y) = \sup \{ |t| : \text{ there exists $\rho \in \pi^{-1}(\hyp_E)$ with $\exp_{t\ham_p}(\rho) \in \pi^{-1}(\hyp_E)$}\}.
\end{equation}
Since $Y$ is compact, $\diam_E(Y)$ is finite for a non-trapping energy $E$. For $I \subset \RR$ we also define
\[
\diam_{I}(Y) = \sup \{ \diam_E(Y): E \in I\}.
\]
This is again finite if $I = [E_0-\delta,E_0+\delta]$ for $E_0$ non-trapping and $\delta>0$ sufficiently small.

\begin{theorem} \label{theo:resonancewidth}
	Let $(X,g)$ and $V \in I^{[-\alpha-1]}(Y)$ be as above. If $E_0 >0$ is non-trapping, then there exists $\delta_0 >0$ with the following property. Given $\delta \in (0,\delta_0)$ and
	\[
	0 < \nu_0 < \frac{\alpha}{\diam_{[E_0-\delta,E_0+\delta]}(Y)},
	\] 
	 there exists $h_0>0$ such that $P$ has no resonance $z$ with
	\begin{equation} \label{eq:window}
	z\in [E_0 - \delta, E_0 + \delta] - i[0,\nu_0 h \log(1/h)]
	\end{equation}
	for $h \in (0,h_0)$. If $\diam_{[E_0-\delta,E_0+\delta]}(Y) = 0$, then the conclusion is valid for any $\nu_0 \in (0,\infty)$.
\end{theorem}

When $V$ is smooth (so that $\alpha$ can be taken arbitrarily large), the fact that Theorem \ref{theo:resonancewidth} holds for arbitrary $\nu_0 \in (0,\infty)$ is originally due to Martinez \cite{martinez2002resonance}.

In Proposition~\ref{prop:asymptotics}, we give an application of (a
slightly more quantitative version of)
Theorem~\ref{theo:resonancewidth} to the asymptotic behavior of
solutions to the time-dependent semiclassical Schr\"odinger equation.

The size of the resonance-free region in Theorem \ref{theo:resonancewidth} is already optimal for the class of compactly supported piecewise-smooth potentials on $\RR$. Consider the operator $P = (hD_x)^2 + V$, where $V$ satisfies the following properties:
\begin{enumerate} \itemsep6pt 
	\item There exists $L > 0$ such that $\supp V \subset [0,L]$.
	\item The restriction of $V$ to $[0,L]$ is smooth.
\end{enumerate}
 If $V$ vanishes to order $k$ at $x=0$ and to order $l$ at $x=L$, then $V \in I^{[-1-\min(k,l)]}(\{0,L\})$. For use in Theorem \ref{theo:existence} below define the quantity
 \[
\phi = (1/2\pi)\arg( V^{(k)}(0^+)\cdot V^{(l)}(L^-)),
 \]
 which arises as the phase shift in a Bohr--Sommerfeld type formula. Consider a spectral interval $[a,b]$, where 
\[
a > \sup V.
\]
Certainly any energy $E \in [a,b]$ is nontrapping for the $\ham_p$ flow. Define (half) the action and period by
\begin{equation}\label{actionandperiod}
S(E) = \int_{0}^{L} (E-V(s))^{1/2}\, ds, \quad T(E) = \int_{0}^{L} \frac{1}{2(E-V(s))^{1/2}} \, ds.
\end{equation}
Of course this is a slight abuse of terminology, since the endpoints are not turning points for the classical dynamics. Observe that 
\[
T(E) = \diam_E(\{0,L\})
\] in the notation of \eqref{eq:diameter}. 

Noting that $S(E)$ is increasing on any interval $[a,b]$ as above, define $[\alpha,\beta] = S([a,b])$ and then set
\[
N(h) = \{ n \in \ZZ: \pi(n  + (l-k)/4 + \phi) \in h^{-1}[\alpha,\beta]\}.
\] 
Note that $\sharp N(h) \sim h^{-1}|\beta-\alpha|/\pi$. We then have the following semiclassical analogue of \cite[Theorem 6]{Zworski4} on the existence and asymptotics of resonances.

\begin{theorem} \label{theo:existence}
Let $k$ and $l$ denote the orders of vanishing of $V$ at $x=0$ and $x=L$, respectively. For each $n \in N(h)$ let 
\[
E_n = S^{-1}(\pi h(n  + (l-k)/4 + \phi)).
\]
There exists $h_0> 0$ such that for each $h\in(0,h_0)$ and $n \in N(h)$ there is a unique resonance $z_n$ satisfying
\begin{equation}\label{1dresonances}
\begin{aligned}
z_n &= E_n - \frac{i(l+k)}{2T(E_n)} h \log (1/h)\\ &+ \frac{ih}{2T(E_n)} \left( \log |V^{(k)}(0^+)\cdot  V^{(l)}(L^-)| - (1/2)(l+k+4) \log(4 E_n) \right) \\ &+ \mathcal{O}(h^2 \log(1/h)^2).
\end{aligned}
\end{equation}
Furthermore, if $M>0$ is sufficiently large, then $\{z_n: n \in N(h)\}$ are all the resonances in $[a,b] + i[-M h\log(1/h),0]$.
\end{theorem}

Because $\pa_E S(E) = T(E)$ and $T(E)$ is uniformly positive for $E \in [a,b]$ it follows that 
\[
|z_n - z_{m}| \geq Ch|n-m|
\]
 for some $C>0$ and each $n,m \in N(h)$. Thus the $z_n$ are all distinct, which gives $h^{-1}|\beta-\alpha|/\pi$ as an asymptotic formula for the number of resonances in $[a,b] + [-iMh\log(1/h),0]$.

\subsection{Context and previous work}

When diffraction of singularities is the only trapping, an increasing
body of work suggests that resonances may occur at
$\Im z \sim -C h \log(1/h)$ for various values of $C>0,$ but no closer
to the real axis.  This is
farther into the lower half-plane than the resonances occurring in cases of elliptic trapping (where they rapidly approach the real axis as $h \rightarrow 0$) or even
for hyperbolic trapped sets, where there is an $\mathcal{O}(h)$ resonance-free
region (see \cite{1209.0843} for references on the subject of
classical dynamical trapping and resonances).  In the diffractive
case, the regularization of the trapped wave with each successive
diffraction is by contrast responsible for the faster rate of decay as measured by
resonance width.

Most of the literature substantiating this heuristic is in the
homogeneous rather than the
semiclassical setting; there the analog of resonances at $E_0-i\nu_0 h \log(1/h)$ are resonances close to a curve
\begin{equation}\label{logcurve}
\Im \lambda \sim -C \log \abs{\Re \lambda}
\end{equation}
as $\smallabs{\lambda} \to \infty.$
That diffraction of singularities \emph{can} in fact create strings of
resonances along such log curves was demonstrated in the homogeneous
setting by Zworski \cite{Zworski4} (see also earlier work by Regge
\cite{Regge:Analytic}).  In the setting of diffraction by analytic
corners in the plane, Burq \cite{Burq:Coin} likewise showed that resonances lie asymptotically on families of curves \eqref{logcurve}
for various values of $C>0.$

In the setting of manifolds with conic singularities, where similar
diffractive propagation occurs, a number of
recent theorems have explored the same theme.
Baskin--Wunsch \cite{BaWu:13} showed on the one hand that
\emph{some} nontrivial region of the form
\[
\Im \lambda > -\nu_0 \log \smallabs{\Re \lambda}, \quad \quad \abs{\lambda}>R
\]
contains no resonances (subject to some genericity conditions on the
relationship among the conic singularities) --- this is analogous to the
gap theorem obtained here.  Galkowski \cite{Ga:15}
then found the largest $\nu_0$ which could
be obtained by the Vainberg parametrix method employed in \cite{BaWu:13}:
\[
\nu_0=(n-1)/(2L_0),
\]
with $n$ the dimension
and $L_0$ the maximal distance between
cone points.  Work of Hillairet--Wunsch using a
trace formula of Ford--Wunsch \cite{FoWu:17} showed that this
constant was in general optimal by proving existence of resonances
with $\Im \lambda \sim -\nu_0 \log \abs{\Re \lambda},$
while \cite{HiWu:17} refined the
description of the resonances on and near this curve.

Closely related work of Datchev--Kang--Kessler \cite{datchev2013nontrapping} studied the distribution of resonances on surfaces of revolution with a cone point and a funnel (as well as other types of infinite ends). The authors find examples where the Laplacian admits resonances with $\Im \lambda \sim -C\log|\Re \lambda|$ as $|\Re \lambda| \rightarrow \infty$, although the classical flow is non-trapping; in such cases the metric is continuous with a conormal singularity.

The results here are, to the best of our knowledge, the first results
on resonances generated by diffractive trapping for semiclassical
Schr\"odinger operators.  The crucial new technical ingredient is the
propagation of singularities results recently obtained by the authors
in \cite{GaWu:18}, which include estimates on the size of the wave of
diffractively reflected singularities.  These singularities bounce back off of even a
mild singularity of $V$ in violation of the most naive application of
the principle of geometric optics, which would say that wavefront set
travels along classical trajectories.

The authors intend in future work to complement the results here,
which show a resonance-free region in all dimensions and existence in
one-dimension, by showing the existence of resonances just below the
gap obtained in Theorem~\ref{theo:resonancewidth}
  in all dimensions, at least in settings where the dynamics is
tractable.  We thus conjecture, based on the evidence of
Theorem~\ref{theo:existence}, that the $\nu_0$
obtained in Theorem~\ref{theo:resonancewidth} is optimal in general
(at least generically).

\subsection*{Acknowledgements}
The authors are grateful to Nicolas Burq and Jeff Galkowski for
helpful discussions, as well as to an anonymous referee for helpful
comments on the manuscript.
OG was partially supported by NSF grant DMS--1502632; JW was partially
supported by NSF grant DMS--1600023.

\section{Resonances}
\subsection{Complex scaling} \label{subsect:complexscaling}

We define resonances of $P$ by the method of complex scaling following \cite[Sections 4.5, 6.2.1]{DyZw:15}. Fix $R_0 > 0$ such that outside of $B(0,R_0)$,
\[
V = 0, \quad g_{ij} = \delta_{ij},
\]
and let $R_1 > R_0$. We define the complex scaled operator $P_\theta$
using a contour $\Gamma_{\theta} = f_\theta(\RR^n)$, where
$f_\theta(x) = x + i \pa_x F(x)$, and $F$ is a smooth convex function
satisfying 
\[
F =0 \text{ near  }B(0,R_1), \quad F = \frac{\tan\theta}{2}|x|^2 \text{ for } |x| \geq 2R_1.
\]
Here we take any fixed $\theta \in (0,\pi/2)$. In particular, the semiclassical principal symbol of $p_\theta$ is given by
\[
p_\theta (x,\xi) = \begin{cases} 
\langle (1 + i \nabla^2 F(x))^{-2}\xi,\xi \rangle, & |x| > R_0, \\
|\xi|_g^2 + V(x) & |x| < R_1.
\end{cases}
\]
Referring to the proof of \cite[Proposition 6.10]{DyZw:15}, we record the following important observation: for an interval $I$,
\begin{equation} \label{eq:flowsagree}
\begin{split} 
\exp_{t \ham_{\Re p_\theta}}(\rho) \in \{ \Im p_\theta = 0\} &\text{ for all } t\in I \\
\Longrightarrow \exp_{t \ham_{\Re p_\theta}}(\rho) = \exp_{t \ham_p}(\rho) &\text{ for all } t\in I. 
\end{split}
\end{equation}

\subsection{The time-dependent Schr\"odinger equation} \label{subsect:tdschrodinger}

Define \[M = \RR_t \times X,\quad Y_M=\RR_t \times Y.\]
As a preliminary step, we pass to the time-dependent semiclassical Schr\"odinger operator on $M$ given by
\[
Q = hD_t + P,
\]
where
$$
D_t = \frac 1 i \frac{\pa}{\pa t}.
$$
If $\tau$ denotes the momentum dual to $t$, then the principal symbol of $Q$ (which does not depend on $t$) is
\[
q(x,\tau,\xi) = \tau + p(x,\xi).
\]
Note that $\tau$ is conserved under the $\ham_q$ flow, and $t$ evolves at unit speed.
We will also work with the complex-scaled operator $Q_\theta = hD_t + P_\theta$, writing $q_\theta = \tau + p_\theta$ for its principal symbol. Since $\Im q_\theta(\tau,x,\xi) = \Im p_\theta(x,\xi)$ for all $\tau$, observe that \eqref{eq:flowsagree} also holds with $q$ replacing $p$.

Given $E \in \RR$, consider the joint energy surface $\{q=0, \,
\tau = -E\}$. In general, there may be non-trivial behavior of
solutions to $Qw=0$ in the characteristic set of $\langle (\tau,\xi)\rangle^{-2} q$ at fiber-infinity in $\overline{T}^*M$, since $Q$ is not elliptic in the non-semiclassical sense. However, we explicitly avoid such issues by restricting to $\tau = -E$. Thus 
\[
(t,x,-E,\xi) \in \{q=0, \, \tau = -E\} \Longleftrightarrow (x,\xi) \in \{p=E\},
\]
with similar observations for $p_\theta$ and $q_\theta$.

We now consider propagation of singularities for the operator $Q$ near $T^*_{Y_M}M$. Define the lift of $\hyp_E$ by 
\[
\hat \hyp_E = \{(t,x',-E,\xi') : (x',\xi') \in \hyp_E\} \subset T^*Y_M,
\]
with the analogous definition for $\hat \gl_E$. 
We also write $\hat \pi : T^*_{m} M \rightarrow T^*_m Y_M$ for the projection whenever $m \in Y_M$.
Adapting the results of \cite{GaWu:18}, we have the following theorem on propagation of singularities for $Q$, stated locally:
\begin{theorem} \label{theo:diffractiveimprovements}
	Let $w = w(h)$ be $h$-tempered in $H^1_{h,\loc}(M)$ such that $Q w = \mathcal{O}(h^\infty)_{L^2_\loc}$.
	\begin{enumerate} \itemsep6pt 
		\item 
		If $q \in \hat \hyp_E$, let $\mu_\pm \in \{q=0, \tau = -E\}$ be the preimages of $q$ under $\hat \pi$ with opposite normal momenta. If $\mu_+ \in \WF^s(w)$ for some $s\in \RR$, then there exists $\varepsilon > 0$ such that
		\[
		\exp_{-t\ham_{q}}(\mu_+) \subset \WF^s(w), \text{ or } \exp_{-t\ham_{q}}(\mu_-) \subset \WF^{s-\alpha}(w),
		\]
		or both, for all $ t\in (0,\varepsilon)$.
		
		\item If $q \in \hat \gl_E$, let $\mu \in \{q=0, \tau = -E\}$ be the unique preimage of $q$ under $\hat \pi$ with vanishing normal momentum. If $\mu \in \WF^s(w)$ for some $s \in \RR$, then there exists $\varepsilon > 0$ such that
		\[
		\exp_{-t\ham_{q}}(\mu) \subset \WF^s(w)
		\] 
		for all $t \in (0,\varepsilon)$.
	\end{enumerate}
\end{theorem}

The (minor) modifications to \cite{GaWu:18} needed to prove Theorem
\ref{theo:diffractiveimprovements} are outlined in Appendix
\ref{appendix}. Note that singularities propagate straight through
$\hat \pi^{-1}(\hat \gl_E)$. Combined with ordinary propagation of
singularities for $Q_\theta$ away from $T^*_{Y_M} M$, we obtain the
following:

\begin{corollary} \label{cor:propagationawayfromhyp}
	Let $w = w(h)$ be $h$-tempered in $H^1_{h,\loc}(M)$ such that $Q_\theta w = \mathcal{O}(h^\infty)_{L^2_\loc}$. Let $s\in \RR$ and $T \geq 0$. If 
 \[
 \mu \in \{q_\theta = 0, \, \tau = -E\}
  \]
 is such that $\exp_{-t\ham_{\Re q_\theta}}(\mu)$ is disjoint from $\hat \pi^{-1}(\hat \hyp_E)$ for each $t \in [0,T]$, then
 \[
 \exp_{-T\ham_{\Re q_\theta}}(\mu) \notin \WF^s(v) \Longrightarrow \mu \notin \WF^s(v).
 \]
\end{corollary}
\begin{proof}
Suppose conversely that $\mu \in \WF^{s}(v)$, and define
\[
T_0 = \sup\{ t \geq 0: \exp_{-t'\ham_{\Re q_\theta}}(\mu) \in \WF^s(v) \text{ for all } t'\in[0,t] \}.
\]
Note that $\mu' = \exp_{-T_0\ham_{\Re q_\theta}}(\mu) \in \WF^s(v)$ since wavefront set is closed. We claim that $T_0 \geq T$, which thus completes the proof. Indeed, if $T_0 < T$, then by hypothesis 
\[
\mu' \notin T^*_{Y_M} M \text{ or } \mu' \in \hat \pi^{-1}(\hat \gl_E).
\]
In the first case $Q_\theta$ is smooth in a neighborhood of $\mu'$, and since $\Im q_\theta \leq 0$ we can apply propagation of singularities \emph{forward} along the $\ham_{\Re q_\theta}$ flow to deduce a contradiction. In the second case $q = q_\theta$ in a neighborhood of $\mu'$, so we can apply the second part of Theorem \ref{theo:diffractiveimprovements} to deduce a contradiction.
\end{proof}

We make the following dynamical definitions, recalling the definition of $R_1$ from Section \ref{subsect:complexscaling}.

\begin{defi} \label{defi:IC} Let $E \in \RR$ and $\rho \in \{p=E\} \setminus \pi^{-1}(\hyp_E)$. We say that $\rho \in \IC_E$ if there exists $T_0 \geq 0$ such that $\exp_{-t\ham_{ p}}(\mu)$ is disjoint from $\pi^{-1}(\hyp_E)$ for each $t \in [0,T_0]$ and
	\[
	|x(\exp_{-t\ham_{p}}(\rho))|\geq 2R_1,\text{ for all } t\geq T_0.
	\]
If $\rho \in \{p=E\} \setminus  \pi^{-1}(\hyp_E)$, then we say that $\rho \in \Gamma_{E}^\pm$ if there exists $t > 0$ such that $\exp_{\mp t\ham_{p}}(\rho) \in \pi^{-1}(\hyp_E)$.
\end{defi} 
Thus, $\IC_E$ consists of points that are ``incoming'' from infinity
without interaction with the hyperbolic region in their past (i.e.,
backward-time flow); they may
or may not interact with the hyperbolic region under forward flow,
however.\footnote{Points in $\IC_E$ may also be far from
  the origin with momentum pointing away from the origin, as long as
  they do not hit $\pi^{-1}(\hyp_E)$ under backward flow;
  these points are irrelevant in practice, owing to ellipticity of the
  complex-scaled operator.}  We could equally well have employed the
definition that $\rho \in \IC_E$ if it escapes to infinity in
backward-time without encountering $\pi^{-1}(\hyp_E).$ The
time $T_0$ built into this quantitative version of the definition is used below, however.

The set $\Gamma_{E}^\pm$
represents points hitting the hyperbolic region in
backward-/forward-time (using notation in loose analogy with that of
unstable/stable manifolds). Note that
\begin{equation} \label{eq:scalingregion}
\{ |x| \geq 2R_1 \} \subset \{\Im p_\theta \neq 0\}.
\end{equation}
The next observation follows immediately from Corollary \ref{cor:propagationawayfromhyp}.

\begin{lemma} \label{lem:ICnoWF}
	Let $w = w(h)$ be $h$-tempered in $H^1_{h,\loc}(M)$ such that $Q_\theta w = \mathcal{O}(h^\infty)_{L^2_\loc}$, and let $ E \in \RR$. If $(x,\xi) \in \IC_E$ and we set $\mu = (t,x,-E,\xi) \in \{q_\theta=0, \, \tau = -E\}$, then
\[
\mu \notin \WF^s(w) \text{ for all } s \in \RR.
\]
\end{lemma}
\begin{proof}
Let $T_0 >0$ be as in Definition \ref{defi:IC}.  It cannot be that $\exp_{-t\ham_{\Re q_\theta}}(\mu) \in \{\Im q_\theta = 0\}$ for all $t \in [0,T_0]$, since using $\Im q_\theta(\tau,x,\xi) = \Im p_\theta(x,\xi)$ and \eqref{eq:flowsagree}, we would conclude that
\begin{equation} \label{eq:qflowsagree}
\exp_{-t\ham_{\Re q_\theta}}(\mu) = \exp_{-t\ham_{q}}(\mu) \text{ for all } t \in [0,T_0],
\end{equation}
and in particular $\exp_{-T_0\ham_{q}}(\mu) \in \{ \Im q_\theta = 0\}$; this contradicts the definition of $T_0$ according to \eqref{eq:scalingregion}.
Thus we can define
\[
T = \inf\{  t \in [0,T_0]: \exp_{-T\ham_{\Re q_\theta}}(\mu) \in \{ \Im q_\theta \neq 0 \} \}.
\]
Note that $\mu'= \exp_{-T\ham_{\Re q_\theta}}(\mu) \notin T^*_{Y_M} M$ since $\Im q_\theta = 0$ in a neighborhood of $T^*_{Y_M} Y$. By semiclassical ellipticity in the smooth setting and the definition of $T$, there exists $\delta>0$ such that 
\[
\mu''= \exp_{-(T+\delta)\ham_{\Re q_\theta}}(\mu) \notin \WF(w),
\]
and hence by forward propagation  in the smooth setting we have $\mu' \notin \WF(w)$ as well.

On the other hand, by the definition of $T$ we must have that \eqref{eq:qflowsagree} holds for all $t \in [0,T]$. Since $(x,\xi) \in \IC_E$, it follows that the backward $\ham_{\Re q_\theta}$ flow from $\mu$ to $\mu'$ is disjoint from $\hat \pi^{-1}(\hat \hyp_E)$, and hence $\mu \notin \WF^s(w)$ for all $s\in \RR$ by Corollary \ref{cor:propagationawayfromhyp}.
\end{proof}

Now we make the assumption that $E$ is a non-trapping energy level.

\begin{lemma} \label{lem:Gammapm}
		Let $w = w(h)$ be $h$-tempered in $H^1_{h,\loc}(M)$ such that $Q_{\theta}w = \mathcal{O}(h^\infty)_{L^2_\loc}$, and let $E >0$ be non-trapping. If there exists 
	\[
	\mu \in \WF^s(w) \cap \{q_\theta = 0, \tau = -E\}
	\]
	for some $s \in \RR$, then there exists $r \in \RR$ and $\mu' = (t,x,-E,\xi) \in \WF^r(w) \cap \{q_\theta = 0\}$ such that 
	\[
	(x,\xi) \in \Gamma_{E}^-.
	\]
\end{lemma}
\begin{proof}
First assume $\mu \notin \hat \pi^{-1}(\hat \hyp_E)$. Arguing precisely as in Lemma \ref{lem:ICnoWF}, the backwards flow $\exp_{-t\ham_{\Re q_{\theta}}}(\mu)$ must encounter a hyperbolic point $\mu_+$ for some $t > 0$. Otherwise,  if 
\[
\mu \in \hat \pi^{-1}(\hat \hyp_E)
\] 
to begin with, simply set $\mu_+ = \mu$. Now let $\mu_-$ project to the same hyperbolic point as $\mu_+$, but with opposite normal momentum. We know from the first part of Theorem \ref{theo:diffractiveimprovements} that there exists $\mu'$ with the requisite properties, obtained by flowing backwards along $\ham_{q}$ from either $\mu_+$ (taking $r = s$) or $\mu_-$ (taking $r =s - \alpha$) for a short time $\varepsilon > 0$, noting that $q =\Re q_{\theta} = q_\theta$ near $\mu_\pm$.
\end{proof}

Note in Lemma \ref{lem:Gammapm} we \emph{assume} that $\mu$ is in both
$\WF^s(w)$ and the characteristic set $\{ q_\theta= 0\}$, since the
inclusion of former in the latter is only guaranteed for a certain
range of $s$ owing to the singularity of $V$ (see \cite[Proposition 7.5]{GaWu:18}).

\begin{lemma} \label{lem:entersIC}
	Let $E > 0$ be non-trapping and $\rho \in \Gamma_{E}^-$. If $\diam_E(Y) > 0$ and $T > \diam_E(Y)$, then
	\[
	\exp_{-T\ham_{p}}(\rho) \in \IC_E.
	\]
	If $\diam_E(Y) = 0$, then $\rho \in \IC_E$.
\end{lemma}
\begin{proof}
	Since $\rho \in \Gamma_{E}^-$, there exists $t_0 > 0$ such that $\rho_0 = \exp_{t_0 \ham_p}(\rho) \in \pi^{-1}(\hyp_E)$.  Clearly if $\diam_E(Y) = 0$, then $\rho \in \IC_E$. Otherwise, if $0 < \diam_E(Y) < T$, set 
	\[
	\rho' = \exp_{-T\ham_{p}}(\rho_0) =  \exp_{(t_0-T)\ham_{p}}(\rho).
	\]
	Then $\exp_{-t\ham_{p}}(\rho')$ is disjoint from $\pi^{-1}(\hyp_E)$ for all $t \geq 0$, and since $E$ is non-trapping the proof is finished.
\end{proof} 

We record a key lemma, which is a refinement of Lemma \ref{lem:Gammapm}.

\begin{lemma} \label{lem:mustdiffract}
Let $w = w(h)$ be $h$-tempered in $H^1_{h,\loc}(M)$ such that $Q_{\theta}w = \mathcal{O}(h^\infty)_{L^2_\loc}$, and let $E > 0$ be nontrapping. If $(x,\xi) \in \Gamma_{E}^-$
and 
\[
\mu  = (0,x,-E,\xi) \in \WF^s(w) \cap \{q_\theta = 0\}
\]
for some $s \in \RR$,
then for all $\delta>0$ there exists $T_0 \leq \diam_E(Y) + \delta$ and 
\[
\mu' = (-T_0,x_0,-E,\xi_0) \in \WF^{s-\alpha}(w) \cap \{q_\theta = 0\}
\] 
such that $(x_0,\xi_0) \in \Gamma_{E}^-$.

\end{lemma}
\begin{proof}
 Define
    \[
  t_0 \equiv \sup\big\{t\geq 0 \colon \exp_{-t' \ham_{\Re q_\theta}} (\mu) \in \WF^s
  (w) \text{ for all } t' \in [0,t] \big\}.
  \]
 We claim that $t_0$ always exists and $t_0 \leq \diam_E(Y)$. To see this, first note that for $t\geq 0$,
  \begin{align*}
  \exp_{-t' \ham_{\Re q_\theta}} (\mu) \in \WF^s
  (w) &\text{ for all }t' \in [0,t] \\ &\Longrightarrow \exp_{-t'\ham_{\Re q_\theta}}(\mu) = \exp_{-t'\ham_{q}}(\mu) \text{ for all }t' \in [0,t].
  \end{align*}
Indeed, if the conclusion fails, then by \eqref{eq:qflowsagree} there must be $t' \in [0,t]$ such that $\exp_{-t'\ham_{\Re q_\theta}}(\mu) \in \Im \{q_\theta \neq 0\}$. But in the neighborhood of such a point $q_\theta$ is smooth, and hence by semiclassical ellipticity $\exp_{-t'\ham_{\Re q_\theta}}(\mu) \notin \WF^s(w)$, which is a contradiction. Furthermore, by Lemmas \ref{lem:ICnoWF} and \ref{lem:entersIC}, if $T > \diam_E(Y)$, then 
  \[
  \exp_{-T\ham_{q}}(\mu) \notin \WF^s(w),
  \]
  which shows that $t < T$.

Now set $\mu_+=\exp_{-t_0 \ham_{\Re q_\theta}} (\mu);$ we have $\mu_+ \in
\WF^s(w)$ since wavefront set is closed.  Also by definition of $t_0$, there is a sequence $\varepsilon_n > 0$ with $\varepsilon_n \rightarrow 0$ such that 
\[
\exp_{-\varepsilon_n \ham_{\Re q_\theta}} (\mu_+) \notin \WF^s (w),
\]
By ordinary propagation of singularities of $Q_\theta$ away from $T^*_{Y_M} Y$ we must have $\mu_+ \in T^*_{Y_M}M$, and by Theorem~\ref{theo:diffractiveimprovements},  $\hat \pi(\mu_+)$ must be a hyperbolic point. Let $\mu_-$
denote the point in $\hat\pi^{-1}(\hat \pi(\mu_+))$ with opposite
normal momentum.  Then Theorem~\ref{theo:diffractiveimprovements} shows
that
\[
\mu' \equiv \exp_{-\delta \ham_{\Re q_\theta}}(\mu_{-}) \in \WF^{s-\alpha}(w)
\]
for $\delta>0$ arbitrarily small. Now $T_0 = -t(\mu') = t_0 + \delta \leq \diam_E(Y) + \delta$ as desired.
\end{proof}

We are now ready to prove Theorem \ref{theo:resonancewidth}.

\subsection{Resonance widths}

We begin by working quite generally, without prejudice as to the operator $P$. Suppose that $u = u(h)$ is $h$-tempered in $H^1_{h,\loc}(X)$. Given a family $z= z(h) \in \CC$, set
\[
\nu(h) = \frac{\Im z(h)}{h\log(h)}.
\]
Now form the functions
\[
w(t,x) = e^{-izt/h}u(x),
\]
and note that a sufficient condition to guarantee that $w = w(h)$ is $h$-tempered in $H^1_{h,\loc}(M)$ is $\nu(h) = \mathcal{O}(1)$. Furthermore, if $(P-z)u = \mathcal{O}(h^\infty)_{L^2_\loc}$, then 
\[
Qw = \mathcal{O}(h^\infty)_{L^2_\loc}.
\]
We need a simple lemma comparing the wavefront of $w$ with its restriction to a fixed time slice $\{t = t_0\}$. For our purposes it will suffice to consider wavefront set away from fiber-infinity. In addition, to the assumption $\nu(h) = \mathcal{O}(1)$, we also assume that $z(h) = E + o(1)$ for some $E \in \RR$.  Note that 
\[
\WF^s(w) \cap T^*M \subset \{\tau = -E\}
\]
for all $s$, since $(hD_t + z)w =0$ and $hD_t + z$ is elliptic when $\tau \neq -E$. Of course this leaves open the possibility that $w$ has wavefront set at fiber-infinity in a direction with $\tau=0$. We need to be slightly more precise: let $\chi_0 \in \CcI(\RR)$ have support near $-E$ and $\chi_1 \in \CcI(\RR)$ have support near $t_0$. Then for each $K \subset X$ compact,
\begin{equation} \label{eq:taulocalization}
\| (1- \chi_0(hD_t))\chi_1(t)w\|_{L^2(\RR \times K)} = \mathcal{O}(h^\infty)
\end{equation}
after integrating by parts.

For each fixed $t_0\in\RR$, introduce the set
\[
W^s(t_0) = \{(x,\xi) \in T^*X : (t_0,x,-E,\xi)  \in \WF^{s}(w)\}.
\]
We have the following lifting lemma:
\begin{lemma} \label{lem:lifting}
	If $z = z(h) \in \CC $ satisfies $z(h) = E + o(1)$ and $|\nu(h)| \leq \nu_0$ for some $E \in R$ and $\nu_0 >0$, then for each $s,t_0 \in \RR$ and $\delta> 0$
	\[
	W^{s-\delta}(t_0) \subset \WF^{s}(h^{\nu(h) t_0}u) \cap T^*X \subset W^s(t_0).
	\]
\end{lemma}
\begin{proof}
	First observe that $\WF^{s}(w(t_0,\cdot)) = \WF^{s}(h^{\nu(h) t_0}u)$ for each fixed $t_0 \in \RR$. Now suppose that $(x_0,\xi_0) \in T^*X$ satisfies
	\[
	(t_0,x_0,-E,\xi_0) \notin \WF^s(w).
	\] 
	For appropriate cutoffs $\phi_0, \phi_1 \in \CcI(X)$, and $\chi_1 \in \CcI(\RR)$ supported near $\xi_0, x_0$, and $t_0$, respectively, \eqref{eq:taulocalization} implies that
	\[
	\phi_0(hD_x)\phi_1(x)\chi_1(t) w = \mathcal{O}(h^s)_{L^2(M)},
	\]
	which immediately implies that $(x_0,\xi_0) \notin \WF^s(w(t_0,\cdot))$. Conversely, suppose that $(x_0,\xi_0) \notin \WF^s(w(t_0,\cdot)) \cap T^*X$, so in the notation above,
	\[
	e^{-izt_0 /h} \phi_0(hD_x)\phi_1(x)u = \mathcal{O}(h^s)_{L^2(X)}.
	\]
	Now if $\chi_1$ as above has sufficiently small support depending on $\delta > 0$, then we can arrange that
	\[
	\phi_0(hD_x)\phi_1(x)\chi_1(t) w = \mathcal{O}(h^{s-\delta})_{L^2(M)}. 
	\]
	Again using \eqref{eq:taulocalization}, this implies that $(x_0,\xi_0) \notin W^{s-\delta}(t_0)$ for each $\delta > 0$.
\end{proof}

Now we place ourselves in the setting of Theorem \ref{theo:resonancewidth}. If there is a resonance in the set 
\[
[E_0 -\delta,E_0 + \delta] - i[0,\nu_0 h \log(1/h)]
\]
for some $\nu_0 > 0$, then we can find a sequence $h_k \rightarrow 0$, complex numbers $z(h_k)$ satisfying
\begin{equation} \label{eq:zconvergence}
\Re z(h_k) = E + o(1), \quad \Im z(h_k) \geq -\nu_0 h_k\log(1/h_k),
\end{equation}
for some $|E - E_0| < \delta$, and eigenfunctions $u(h_k) \in L^2(X)$ of the complex-scaled operator,
\[
(P_\theta - z(h_k))u(h_k) = 0.
\]
Throughout, we suppress the index $k$. 

Normalize $u$ by $\| u \|_{L^2} = 1$. In order to eventually apply Theorem \ref{theo:diffractiveimprovements}, we verify that the $H^1_h$ (in fact, the $H^2_h$) norms of $u$ are also uniformly bounded in $h$. This follows from standard semiclassical analysis by observing that $P_\theta - V$ is smooth, that
\[
|\sigma_h(P_\theta -V)(x,\xi)| \geq \delta \langle \xi \rangle^2 - C_0
\]
uniformly on $X$, and that $V \in L^\infty(X)$. 

Observe that $\WF^s(u) \neq \emptyset$ for each $s > 0$, since $\| u \|_{L^2} = 1$. Taking $s = \alpha$, we conclude that \[
\WF^\alpha(u) \subset \{\langle \xi \rangle^{-2}(p_\theta-E) = 0\}
\] 
Indeed, away from $Y$ this is ordinary semiclassical elliptic 
regularity (hence applies with any $s$), whereas near $T^*_Y X$, where $p_\theta=p$, we can apply
\cite[Proposition 7.5]{GaWu:18} (whose proof applies nearly verbatim even when $z$ is not real valued).  Of course since $p_\theta$ is elliptic at fiber-infinity, this can be rewritten as
 \[
 \WF^\alpha(u) \subset \{p_\theta =E\}.
 \]

Next, we let $w = e^{-izt/h}u$ with $u = u(h)$ our family of eigenfunctions as above, and observe that Lemma \ref{lem:lifting} applies to the family $w = w(h)$.

\begin{proof}[Proof of Theorem \ref{theo:resonancewidth}]
	
	As noted above, there exists $(x_0,\xi_0) \in \WF^\alpha(u) \cap \{p_\theta =E\}$. Let 
	\[
	\mu_0 = (0,x_0,-E,\xi_0) \in \{q_\theta =0, \tau = -E \}
	\]
	and note that $\mu_0 \in \WF^{\alpha}(w)$ by Lemma \ref{lem:lifting}. Applying Lemmas \ref{lem:Gammapm} and \ref{lem:lifting} we conclude there exists
\[
\rho \in \WF^r(u) \cap \Gamma_{E}^-
\]
for some $r \in \RR$.
This is already a contradiction if $\diam_E(Y) = 0$, since by Lemmas \ref{lem:ICnoWF}, \ref{lem:entersIC}, and \ref{lem:lifting} we have $\rho \in \IC_E$ and hence $\rho \notin \WF^r(u)$.

Otherwise, we derive a lower bound for $\nu_0$ as follows. Since $\WF^0(u) = \emptyset$, we may define
	\begin{equation} \label{eq:s0}
	s_0=\inf \{ s:  \WF^s(u) \cap \Gamma_{E}^- \neq \emptyset\}.
	\end{equation}
	Pick any $s > s_0$, so there is $(x,\xi) \in  \WF^s(u) \cap \Gamma_{E}^-$, and apply Lemma \ref{lem:mustdiffract} to $\mu = (0,x,-E,\xi)$, which by Lemma \ref{lem:lifting} indeed satisfies
	\[
\mu \in \WF^s(w) \cap \{ q_\theta = 0\}.
	\]
	Thus for each $\delta > 0$ we can find $T < \diam_E(Y)+\delta$ such that 
	\[
	\mu' = (-T,x_1,-E,\xi) \in \WF^{s-\alpha}(w) \cap \{q_\theta = 0\}
	\]
	and $(x_1,\xi_1) \in \Gamma_{E}^-$. Furthermore, by Lemma
        \ref{lem:lifting}, $(x_1,\xi_1) \in \WF^{s-\alpha+\nu_0 T + \delta}(u)$, and hence
	\[
	(x_1,\xi_1) \in \WF^{s-\alpha + \nu_0 T + \delta}(u) \cap \Gamma_{E}^-.
	\]
By definition of $s_0$, we must have $s-\alpha + \nu_0 T + \delta > s_0$. But $s > s_0$ and $\delta>0$ are both arbitrary, which implies that 
\[
\nu_0 \geq \alpha/\diam_E(Y).
\] 	
This completes the proof of the Theorem \ref{theo:resonancewidth}.
\end{proof}

\subsection{Applications to quantum evolution}

By a slight modification of its proof, Theorem \ref{theo:resonancewidth} also provides bounds on the cut-off resolvent (or rather its meromorphic continuation) in the resonance-free region given by \eqref{eq:window} (see Lemma \ref{lem:resolventbound} below). This implies that under time evolution, the $L^2$ norm of quantum states which are frequency localized near $E_0$ decay at rate $h^{\nu_0 t}$. For the following result we adopt the notation of Theorem \ref{theo:resonancewidth}.

\begin{lemm} \label{lem:resolventbound}
Given $\chi \in \CcI(X)$, there exist $C_0 > 0$ such that
\begin{equation} \label{eq:resolventbound}
\| \chi (P-z)^{-1} \chi \|_{L^2 \rightarrow L^2 } \leq h^{-C_0}
\end{equation}
for each $z \in [E_0 -\delta, E_0 + \delta] - i[0,\nu_0 h \log(1/h)]$ and $h \in (0,h_0)$.
\end{lemm}
\begin{proof}
We begin with a preliminary observation: although for simplicity the propagation results for $Q_\theta$ used previously were stated for $h$-tempered families $w= w(h)$ satisfying $Q_\theta w = \mathcal{O}(h^\infty)_{L^2_\loc}$, in actuality propagation of $\WF^s(w)$ only requires that $\WF^{s+1}(Q_\theta w) = \emptyset$ (see the more precise statements throughout \cite{GaWu:18}). 

Recall that $\chi (P-z)^{-1} \chi = \chi (P_\theta - z)^{-1}\chi$ when the scaling region is chosen appropriately depending on $\chi$ (see \cite[Theorem 4.37]{DyZw:15}). Suppose that \eqref{eq:resolventbound} does not hold in the region \eqref{eq:window}. Then there exist $h_k > 0$ and $z(h_k) \in \CC$ satisfying \eqref{eq:zconvergence}, as well as $u(h_k) \in L^2(X)$ such that
\[
\| u(h_k) \|_{L^2} = 1, \quad \| (P_\theta - z(h_k))u(h_k) \|_{L^2} \leq h_k^{C_k},
\]
where $0 < C_k \rightarrow \infty$.
The idea is to obtain a lower bound on $\nu_0$ by repeating the argument of Theorem \ref{theo:resonancewidth} but with the family $u = u(h)$ as above (again we suppress the index $k$). 

Let $s_0$ be defined as in \ref{eq:s0}, and pick any $s > s_0$. After forming the functions $w = e^{-izt/h}u$, the proof of Theorem \ref{theo:resonancewidth} proceeds by applying propagation of $\WF^{s}(w)$. While it is not true that $\WF^{s+1}(Q_\theta w) = \emptyset$ globally on $M$, it is true over finite time intervals for $k$ sufficiently large; since we only use propagation for $Q_\theta$ on a fixed finite time interval (related to the diameter of $Y$), the same argument applies, showing that $\nu_0 \geq \alpha /\diam_{E}(Y)$. 
\end{proof} 

Again using the notation of Theorem \ref{theo:resonancewidth}, we provide asymptotics for the Schr\"odinger propagator.

\begin{prop}\label{prop:asymptotics}
	Let $\chi \in \CcI(X)$ and $\psi \in \CcI((E_0-\delta,E_0+\delta))$. There exists $C, T_0 > 0$ such that if $t \geq T_0$, then
	\[
	\| \chi e^{-itP/h} \psi(P) \chi \|_{L^2 \rightarrow L^2} \leq Ch^{\nu_0 t}.
	\]
\end{prop}
\begin{proof}
The proof is the same as in \cite[Theorem 7.15]{DyZw:15}, which is based on a contour deformation in the spectral representation of $e^{-itP/h} \psi(P)$ to a horizontal line of depth $Mh\log(1/h)$; in the given reference $M > 0$ can be taken arbitrary, but here we take 
\[
M = \nu_0 + \varepsilon,
\] 
where $\varepsilon > 0$ is such that $M < \alpha/\diam_{[E_0 -\delta,E_0+\delta]}(Y)$. The resulting bound is
\[
\| \chi e^{-itP/h} \psi(P) \chi \|_{L^2 \rightarrow L^2} \leq Ch^{M t -M'} + \mathcal{O}(h^\infty \langle t \rangle^{-\infty}).
\]
for some $M' > 0$. The result follows for $t > 0$ sufficiently large.
\end{proof}

\section{Existence of resonances in one dimension}

\subsection{WKB solutions}
Throughout this section we adopt the notation of Theorem \ref{theo:existence} and the paragraph preceding it. Since $\supp V \subset [0,L]$, a complex number $z \in \CC$ is a resonance of $P = (hD_x)^2 + V$ precisely if it satisfies
\begin{equation}
\begin{cases} \label{eq:outoingcondition}
(P-z)u = 0, \\
hD_x u (0) + z^{1/2} u(0) =0,\\
hD_x u (L) - z^{1/2} u(L) =0.
\end{cases}
\end{equation}
Fix an interval $[a,b]$ with $a > \sup V$. Given $M > 0$, set
\[
\Omega_M(h) = [a,b] + i[-Mh \log(1/h),0].
\]
The fact that $I = [0,L]$ is a classically allowed region implies the following:
\begin{lemma} \label{lem:polynomialgrowth}
If $u$ solves the equation $(P-z)u = 0$ with initial conditions 
\[
u(0) =\mathcal{O}(1), \quad  hD_x u(0) = \mathcal{O}(1),
\] 
then $u$ and $\pa_x u$ are polynomially bounded on $I = [0,L]$ uniformly in $z \in \Omega_M(h)$
\end{lemma}
\begin{proof}
 This follows from a semiclassical version of the energy estimates in \cite[Section 23.2]{Hormander3}; cf. \cite[Lemma E.60]{DyZw:15}. At most polynomial growth in $h$ arises from the fact that the imaginary part of $z$ can be of size $\mathcal{O}(h\log(1/h))$.
\end{proof}

Next we consider approximate WKB solutions to $(P-z)u=0$. Following \cite{berry1982semiclassically}, for this problem it is convenient to consider approximate solutions in exponential form. Define 
\[
\psi_0 = (z- V)^{1/2},
\]
and then set $\psi_{+,0} = \psi_{-,0} = \psi_0$. For $i \geq 1$, define $\psi_{\pm,i}$ recursively by 
\begin{equation} \label{eq:recursion}
\psi_{\pm,k}(x) = \pm \frac{i}{2\psi_0(x)}\pa_{x}\psi_{\pm,k-1}(x) - \frac{1}{2\psi_0(x)}\sum_{j=1}^{k-1}\psi_{\pm,j}(x)\psi_{\pm,k-j}(x).
\end{equation}
Each $\psi_{\pm,k}$ is smooth on $I$ and depends holomorphically on $z \in [a,b] + i[-C_0,0]$. Let $\psi_\pm$ be a function admitting an asymptotic expansion
\[
\psi_\pm \sim \sum_{j=0}^\infty  h^j \psi_{\pm,j}
\]
and depending holomorphically on $z$. We then set 
\[
\varphi_\pm (x)= \int_0^x \psi_\pm(s) \, ds, \quad u_\pm = e^{\pm i\varphi_{\pm}/h}.
\]
 If $z \in \Omega_M(h)$, then $u_\pm$ are polynomially bounded on $I$, and hence the recursion relation \eqref{eq:recursion} guarantees that
\[
(P-z)u_\pm = \mathcal{O}(h^\infty)_{\CI(I)}, \quad z \in \Omega_M(h)
\]
for any fixed $M>0$. The usefulness of this exponential form comes from the following observation (cf. \cite[Appendix 2]{berry1982semiclassically}).

\begin{lemma}
For each $j$,
\[
\psi_{\pm,j}(x) = \left( \frac{\pm i}{2\psi_0(x)}\right)^j \psi^{(j)}_0(x) + F_j(\psi_0(x), \psi_0'(x),\ldots,\psi_0^{(j-1)}(x))
\]
for a smooth function $F_j(t_0,t_1,\ldots,t_{j-1})$ such that $F_j(t_0,0,\ldots,0) = 0$ for all $t_0 \in \RR$. 
\end{lemma} 
\begin{proof}
This follows by induction from the recursion relations \eqref{eq:recursion}.
\end{proof}

We then obtain the following corollary: if $V$ vanishes to order $k \geq 1$ at a point $x_0 \in I$, then $\psi_{\pm,j}(x_0) = 0$ for $j = 1,\ldots,k-1$, and
\begin{equation} \label{eq:psiatvanishing}
\psi_{\pm,k}(x_0) = -i^{\pm k} (2z^{1/2})^{-k-1} V^{(k)}(x_0).
\end{equation}
If $x_0$ is an endpoint of $I = [0,L]$, then the same formulas hold in the sense of one-sided limits. Indeed,
\[
V(x) = c_0(x -  x_0)^k + \mathcal{O}(|x-x_0|^{k+1}), \quad c_0 =  V_0^{(k)}(x_0)/k!,
\]
which shows that
\[
\psi_0(x) = z^{1/2} - (c_0/2)z^{-1/2} (x-x_0)^k + \mathcal{O}(|x-x_0|^{k+1}).
\]
The formula \eqref{eq:psiatvanishing} follows immediately from this expression.
\subsection{Outgoing condition}
Observe that $u_\pm(0) = 1$ and $hD_x u_\pm(0) = \pm \psi_\pm(0)$, so if we form the function
\[
v = u_- - \frac{z^{1/2} - \psi_-(0)}{z^{1/2} + \psi_+(0)} u_+,
\]
then $hD_x v(0) = -z^{1/2}v(0)$, and $v$ is polynomially bounded on $I$ for $z \in \Omega_M(h)$.

\begin{lemma} \label{lem:wronskian}
Let $u$ solve the equation $(P-z)u = 0$ with initial data $u(0) = v(0)$ and $hD_x u(0) = -z^{1/2}u(0)$. Then
\[
u = v + \mathcal{O}(h^\infty)_{\CI(I)}
\] 
uniformly for $z \in \Omega_M(h)$.
\end{lemma} 
\begin{proof}
This follows from the identity 
\[
u = \frac{\mathcal{W}_h(u,u_+)u_ - - \mathcal{W}_h(u,u_-)u_+}{\mathcal{W}_h(u_-,u_+)},
\]
where $\mathcal{W}_h(f,g) = f \cdot  (h\pa_x g) - (h\pa_x f)\cdot g$ is the semiclassical Wronskian. Indeed, the fact that $u$ and $u_\pm$ are polynomially bounded on $I$ for $z \in \Omega_M(h)$ implies that all Wronskians appearing in the formula above are constant modulo $\mathcal{O}(h^\infty)$. It is then a straightfoward computation of the Wronskians at $x=0$ using the specified initial conditions.
\end{proof} 	

Fix $M > 0$. In view of \eqref{eq:outoingcondition} and Lemma \ref{lem:wronskian}, resonances $z \in \Omega_M(h)$ are characterized as solutions to an equation of the form
\begin{equation} \label{eq:qc}
\left(\frac{\psi_+(L) - z^{1/2}}{\psi_+(L) + z^{1/2}} \right) \left( \frac{\psi_-(0) - z^{1/2}}{\psi_-(0)+z^{1/2}} \right)  e^{i\varphi_+(L)/h} -e^{-i\varphi_-(L)/h} = \mathcal{O}(h^\infty).
\end{equation}
Both sides of this equation are holomorphic in $z \in \Omega_M(h)$. We replace this with a simpler expression by inserting the asymptotics of $\psi_\pm$, making sure to only incur errors that are holomorphic in $z \in \Omega_M(h)$; we will continue to use ordinary Landau notation to denote these errors. First, note that according to \eqref{eq:psiatvanishing},
\begin{align*}
&\psi_+(L) - z^{1/2} =  -i^l h^l(2z^{1/2})^{-l-1}V^{(l)}(L^-)(1+\mathcal{O}(h))\\
&\psi_-(0)  - z^{1/2}  =  -i^{-k} h^k (2z^{1/2})^{-k-1}V^{(k)}(0^+)(1+\mathcal{O}(h)), \\
&\psi_+(L) + z^{1/2} = 2z^{1/2} + \mathcal{O}(h^l),\\ 
&\psi_-(0) + z^{1/2} = 2z^{1/2} + \mathcal{O}(h^k).
\end{align*}
We also multiply \eqref{eq:qc} through by $e^{i\varphi_-(L)/h} =
\mathcal{O}(h^{-N})$ (for some $N \in \NN$). Define the phase function
\begin{equation} \label{eq:phase}
\varphi(x) = \int_0^x (z-V(s))^{1/2}\, ds,
\end{equation}
and observe that $\varphi_+(L) + \varphi_-(L) = 2\varphi(L) + \mathcal{O}(h^2)$ since $\psi_{+,1} = -\psi_{-,1}$. In particular, 
\[
e^{i(\varphi_+(L) + \varphi_-(L))/h} = e^{2i\varphi(L)/h}(1+ \mathcal{O}(h)).
\]
Inserting this information into \eqref{eq:qc}, we obtain an equivalent equation 
\[
i^{l-k} h^{l+k} (2z^{1/2})^{-l-k-4}(V^{(k)}(0^+)\cdot V^{(l)}(L^-))e^{2i\varphi(L)/h} - 1 =\mathcal{O}(h).
\]
This equation already implies the necessity of Theorem
\ref{theo:existence} (i.e., the fact that resonances in
$\Omega_M(h)$ may only be of the form \eqref{1dresonances}) by
considering the modulus and argument of this equation.  To show
existence of these resonances, consider the function
\[
F(w,E,h) = i^{l-k} h^{l+k} (2E^{1/2})^{-l-k-4}(V^{(k)}(0^+)\cdot V^{(l)}(L^-)) e^{2i(S(E) + wT(E))/h} - 1.
\]
Here $E \in [a,b]$ is treated as a parameter, and $F(w,E,h)$ is
holomorphic in $w \in \CC$. We have used the notation $S(E)$ and
$T(E)$ from Theorem \ref{theo:existence}. Note that if $z=E+w \in
\Omega_M(h),$ then
\[
\varphi(L) = S(E) + wT(E) + \mathcal{O}(|w|^2).
\]
If $|w|$ and $|w|^2/h$ are both small, then $z$ is resonance if and only if $w$ satisfies an equation of the form
\begin{equation} \label{eq:resonanceequationF}
F(w,E,h)  = \mathcal{O}(h + |w| + |w|^2/h)
\end{equation}
with both sides holomorphic in $w$. Now if $n \in N(h)$ and $E_n$ are as in Theorem \ref{theo:existence}, then $F(w_n,E_n,h) = 0$ for the choice
\begin{multline*}
w_n =  \frac{-ih}{2T(E_n)} (l+k)\log (1/h)\\ + \frac{ih}{2T(E_n)} \left( \log |V^{(k)}(0^+) V^{(l)}(L^-)| - (1/2)(l+k+4) \log(4 E_n) \right).
\end{multline*}
Furthermore, $\pa_w F(w_n,E_n,h) = (2i/h)T(E_n)$, and given $\varepsilon >0$ there exists $C>0$ such that
\[
n \in N(h) \text{ and } |w-w_n| \leq \varepsilon h \Longrightarrow |\pa_w^2 F(w,E_n,h)| \leq C h^{-2}.
\]
Thus by Taylor's theorem, for any $A> 0$ we can find $h_0, C_0 > 0$ such that 
\[
|F(w,E_n,h)| \geq A h \log(1/h)^2
\]
for all $h \in (0,h_0)$ and $n \in N(h)$ whenever $|w-w_n| = C_0 h^2 \log(1/h)^2$. On the other hand, there exists $B>0$ independent of $C_0$ such that right hand side of \eqref{eq:resonanceequationF} is bounded by $B h \log (1/h)^2$ whenever $|w-w_n| = C_0 h^2 \log(1/h)^2$ and $h \in (0,h_0)$ uniformly in $n \in N(h)$, shrinking $h_0>0$ depending on $C_0$ if necessary. Thus we first fix $A>B$, choose $h_0,C_0 > 0$ as above, and then apply Rouch\'e's theorem. It follows that for each $h\in (0,h_0)$ and $n \in N(h)$ there exists a unique resonance $z_n \in \Omega_M(h)$ satisfying
\[
z_n = E_n + w_n + \mathcal{O}(h^2\log(1/h)^2),
\]
thus completing the proof of Theorem \ref{theo:existence}. \qed

\appendix 

\section{Proof of Theorem \ref{theo:diffractiveimprovements}} \label{appendix}

In \cite{GaWu:18}, propagation of singularities for operators of the
form $P = h^2\Lap + V$ with $V$ real-valued was discussed. Although
$Q$ is not elliptic at fiber-infinity (unlike $P$), since we are only
considering propagation of singularities in compact subsets of $T^*M$
viewed as the interior of $\overline{T}^*M$, there is little
difference in the proofs. There are two ingredients:
\begin{enumerate} \itemsep6pt 
	\item Propagation of singularities along generalized broken bicharacteristics ($\GBB$s).
	\item Diffractive improvements at hyperbolic and glancing points.
\end{enumerate}

The main difference is in the propagation arguments along $\GBB$s. The preliminary material in \cite[Section 4]{GaWu:18} goes through unchanged provided $p$ is replaced with
\[
q = \tau + |\xi|^2_g + V.
\]
In the neighborhood of a point $m \in Y_M$ we can use coordinates $(t,x^1,x')$, where $(x^1,x')$ are normal coordinates on $Y$ with respect to $g$. We then replace $\tilde P$ with $\tilde Q = Q - (hD_{x_1})^*(hD_{x_1})$ in these local coordinates. 

The elliptic, hyperbolic, and glancing sets are defined in the obvious way; we denote these $\hat \ellip, \hat \hyp, \hat \gl$. In the notation of Section \ref{subsect:tdschrodinger},
\[
\hat \hyp_E = \hat \hyp \cap \{\tau=-E\}, \quad \hat \gl_E = \hat \gl \cap \{\tau=-E\}.
\]
Note that $\hat \hyp \cup \hat \gl$ may intersect fiber-infinity of $\overline{T}^* Y_M$ in directions with $\tau \neq 0$, but of course $\hat \hyp_E \cup \hat \gl_E$ are compact subsets of $T^*Y_M$.

The study of the elliptic set $\hat \ellip$ in \cite[Section 5.1]{GaWu:18} needs only minor modifications provided we stay away from fiber infinity; this amounts to working with compactly microlocalized b-pseudodifferential operators only.

 The most important point is to replace the Dirichlet form associated to $P$ in \cite[Lemma 5.3]{GaWu:18} with the expression
\[
\int_M h^2|d_X Aw|_g^2 + V|Aw|^2 + hD_t Aw \cdot \overline{Aw} \, d\tilde g,
\]
where $d_X$ is the differential on $X$ lifted to $M = \RR \times X$ by the product structure, and $\tilde g$ is the product metric $dt^2 + g$. The proof of \cite[Lemma 5.3]{GaWu:18} applies without change. Modifying the Dirichlet form affects \cite[Eq. 5.3]{GaWu:18}; in the notation there, the relevant replacement is an estimate of the form
\begin{multline*}
\| Aw \|_{H^1_h} \leq C \| GQw\|_{H^{-1}_h} + Ch \|Gw\|_{H^1_h} 
\\ + C_0(\|Aw\|_{L^2} + \|hD_t Aw \|_{L^2}) + \mathcal{O}(h^\infty)\| w \|_{H^1_h},
\end{multline*}
with $C_0 > 0$ independent of $A$, namely one must add $\|D_tAw\|_{L^2}$ to the right hand side. On the other hand, if $A$ has compact b-microsupport, then we can estimate
\[
\| hD_t Aw\|_{L^2} \leq C_1\| Aw \|_{L^2} + \mathcal{O}(h^\infty)\| w \|_{L^2},
\]
where $C_1 > 0$ depends only on the size of the b-microsupport of $A$; thus the important estimate \cite[Eq. 5.3]{GaWu:18} is still valid. The proofs of \cite[Lemmas 5.5, 5.6]{GaWu:18} go through, provided $V$ is replaced with $V+ \tau$. With these modifications, b-elliptic regularity (\cite[Proposition 5.2]{GaWu:18}) continues to hold, at least away from fiber-infinity.

The analysis at $\hat \hyp$ is essentially unchanged, provided one
only consider propagation of singularities away from fiber-infinity;
this is certainly the case near $\hat \hyp_E$. One needs only to
account for the additional localization in $(t,\tau)$, and in the
third step of the proof of \cite[Lemma 5.10]{GaWu:18} we make the
replacement
\begin{equation} \label{eq:replace}
h^2\Lap_k \Longrightarrow h^2 \Lap_k + hD_t.
\end{equation}
In that case, \cite[Proposition 5.8]{GaWu:18} still holds.

Finally, we come to the analogues of \cite[Theorems 2 and 3]{GaWu:18},
which are proved using ordinary pseudodifferential operators. Again,
since in Theorem \ref{theo:diffractiveimprovements} only points $\mu$
in the finite parts of the fiber of $T^*M$ (rather than fiber-infinity
in the compactification) are considered, the proofs in \cite[Section
7]{GaWu:18} are still valid provided we make the replacement
\eqref{eq:replace}.

\bibliographystyle{plain} 
\bibliography{all}
  
\end{document}